\documentclass[reqno]{amsart}
\usepackage{amsfonts}
\usepackage{amsmath}
\usepackage{amssymb}

\newtheorem{theorem}{Theorem}

\newtheorem{definition}[theorem]{Definition}

\newtheorem{lemma}[theorem]{Lemma}

\newtheorem{remark}[theorem]{Remark}

\begin{document}
\title[New Identities for Padovan Numbers]{New Identities for Padovan Numbers}
\author{Gamaliel Cerda-Morales}
\address{Departamento de Matem\'atica, Facultad de Ciencias F\'isicas y Matem\'aticas, Universidad de Concepci\'on, Casilla 160-C, Concepci\'on, Chile.}
\email{gamaliel.cerda@usm.cl}
\urladdr{http://orcid.org/0000-0003-3164-4434}


\begin{abstract}
In \cite{Choi-Jo}, the $am+b$ ($0\leq b<a$) subscripted Tribonacci numbers are studied. This work is devoted to study a new generalization of Fibonacci numbers called Padovan numbers. In particular, the $am+b$ subscripted Padovan numbers will be expressed by three $a$ step apart Padovan numbers for any $0\leq b<a$, where $a\in \mathbb{Z}$.
\end{abstract}

\maketitle

\section{Introduction}
In Fibonacci numbers, there clearly exists the term golden ratio which is defined as the ratio of two consecutive of Fibonacci numbers that converges to $\alpha=\frac{1+\sqrt{5}}{2}$. In a similar manner, the ratio of two consecutive Padovan 
numbers converges to $$\rho=\sqrt[3]{\frac{1}{2}+\frac{1}{6}\sqrt{\frac{23}{3}}}+\sqrt[3]{\frac{1}{2}-\frac{1}{6}\sqrt{\frac{23}{3}}}\in \mathbb{R}.$$

There are so many studies in the literature that concern about the generalized Fibonacci numbers such as Tribonacci, third-order Jacobsthal and Padovan numbers (see, for example \cite{Cer,Cer1,Choi-Jo,Fen,Irm-Alp,Sha-And-Hor}). Although the study of Padovan numbers started in the beginning of 19 century under different names, the master study was published in 2006 by Shannon et al. in \cite{Sha-And-Hor}. In the above paper, the authors defined the Padovan $\{P_{n}\}_{n\geq0}$ sequence as 
\begin{equation}
P_{n+3}=P_{n+1}+P_{n},\ P_{0}=P_{1}=P_{2}=1.
\end{equation}

A number of properties of the Tribonacci sequence were studied in \cite{Choi-Jo} by Tribonacci table method. Let $a\in \mathbb{N}$ and $m\geq0$, in this work we shall generalize the identities about $am$ subscripted Padovan numbers $P_{am}$ to any $P_{am+b}$ ($1\leq b\leq a$). One of our main theorem is to express $P_{am+b}$ by $P_{2a+b}$, $P_{a+b}$ and $P_{b}$, which are $a$ step apart terms.

\section{Padovan table}
For $a\in \mathbb{N}$, when we say $a$ columns Padovan table we mean a rectangle shape having $a$ columns that consists of the all Padovan numbers from $P_{1}$ in order. So,
\begin{equation*}
\left|
\begin{array}{cccc}
P_{1} & P_{2} & ... & P_{a} \\ 
P_{a+1} & P_{a+2} & ... & P_{2a} \\ 
P_{2a+1} & P_{2a+2} & ... & P_{3a} \\ 
... & ... & ... & ...%
\end{array}%
\right| .
\end{equation*}%
We shall investigate a third-order linear recurrence $P_{n}=\rho
_{a}P_{n-a}+\sigma _{a}P_{n-2a}+P_{n-3a}$ for Padovan numbers with some $%
\rho _{a},\sigma _{a}\in \mathbb{Z}$. In the next result, we show some known results for Padovan numbers using the inductive method.

\begin{lemma}\label{lem1} 
Let $P_{n}$ the $n$-th Padovan number, then $P_{n}=2P_{n-2}-P_{n-4}+P_{n-6},$ $P_{n}=3P_{n-3}-2P_{n-6}+P_{n-9}$ and $P_{n}=2P_{n-4}+3P_{n-8}+P_{n-12}.$
\end{lemma}
\begin{proof}
Observe that $P_{6}=4=2\cdot 2-1+1=2P_{4}-P_{2}+P_{0}$. Then, if we assume $%
P_{t}=2P_{t-2}-P_{t-4}+P_{t-6}$ for all $t<n$ then 
\begin{eqnarray*}
\bigskip P_{n} &=&P_{n-2}+P_{n-3} \\
&=&\left( 2P_{n-4}-P_{n-6}+P_{n-8}\right) +\left(
2P_{n-5}-P_{n-7}+P_{n-9}\right) \\
&=&2\left( P_{n-4}+P_{n-5}\right) -\left( P_{n-6}+P_{n-7}\right) +\left(
P_{n-8}+P_{n-9}\right) \\
&=&2P_{n-2}-P_{n-4}+P_{n-6}.
\end{eqnarray*}%
Similar to this, we notice $P_{9}=9=3\cdot 4-2\cdot
2+P_{0}=3P_{6}-2P_{3}+P_{0}.$ If we assume $P_{t}=3P_{t-3}-2P_{t-6}+P_{t-9}$
for all $t<n$, then the induction hypothesis proves $%
P_{n}=3P_{n-3}-2P_{n-6}+P_{n-9}.$ 

Analogously, since $P_{12}=21=2\cdot
7+3\cdot 2+1=2P_{8}+3P_{4}+P_{0},$ the identity $%
P_{n}=2P_{n-4}+3P_{n-8}+P_{n-12}$ can be proved immediately by induction.
\end{proof}

\begin{remark}
Note that the identity $P_{4m}=2P_{4(m-1)}+3P_{4(m-2)}+P_{4(m-3)}$ is a
special case of $P_{n}=2P_{n-4}+3P_{n-8}+P_{n-12}$ in above lemma when $n$ is divisible by 4. In general, the identity $P_{am}=\rho _{a}P_{a(m-1)}+\sigma_{a}P_{a(m-2)}+P_{a(m-3)}$ is a special case of $P_{n}=2P_{n-4}+3P_{n-8}+P_{n-12}$ when $n$ is divisible by $a$. We extend Lemma \ref{lem1} to any
integer $0\leq a\leq 8.$
\end{remark}
\begin{theorem}
\label{teo1} Let $0\leq a\leq 8$. Then, the third-order recurrence $P_{n}=\rho _{a}P_{n-a}+\sigma _{a}P_{n-2a}+P_{n-3a}$ of $P_{n}$\ holds with
the following $(\rho _{a},\sigma _{a}).$%
\begin{equation*}
\begin{tabular}{|l|l|l|}
\hline
$a$ & $(\rho _{a},\sigma _{a})$ & $%
\begin{array}{cccc}
P_{n}= & \rho _{a}P_{n-a} & +\sigma _{a}P_{n-2a} & +P_{n-3a}%
\end{array}%
$ \\ \hline
$1$ & $(0,1)$ & $%
\begin{array}{cccc}
P_{n}= &  & P_{n-2} & +P_{n-3}%
\end{array}%
$ \\ 
$2$ & $(2,-1)$ & $%
\begin{array}{cccc}
P_{n}= & 2P_{n-2} & -P_{n-4} & +P_{n-6}%
\end{array}%
$ \\ 
$3$ & $(3,-2)$ & $%
\begin{array}{cccc}
P_{n}= & 3P_{n-3} & -2P_{n-6} & +P_{n-9}%
\end{array}%
$ \\ 
$4$ & $(2,3)$ & $%
\begin{array}{cccc}
P_{n}= & 2P_{n-4} & +3P_{n-8} & +P_{n-12}%
\end{array}%
$ \\ 
$5$ & $(5,-4)$ & $%
\begin{array}{cccc}
P_{n}= & 5P_{n-5} & -4P_{n-10} & +P_{n-15}%
\end{array}%
$ \\ 
$6$ & $(5,2)$ & $%
\begin{array}{cccc}
P_{n}= & 5P_{n-6} & +2P_{n-12} & +P_{n-18}%
\end{array}%
$ \\ 
$7$ & $(7,1)$ & $%
\begin{array}{cccc}
P_{n}= & 7P_{n-7} & +P_{n-14} & +P_{n-21}%
\end{array}%
$ \\ 
$8$ & $(10,-5)$ & $%
\begin{array}{cccc}
P_{n}= & 10P_{n-8} & -5P_{n-16} & +P_{n-24}%
\end{array}%
$ \\ \hline
\end{tabular}%
\end{equation*}
\end{theorem}
\begin{proof}
Clearly $P_{n}=P_{n-2}+P_{n-3}$ shows $(\rho _{1},\sigma _{1})=(0,1).$ And
above Lemma shows $(\rho _{a},\sigma _{a})=(2,-1),$ $(3,-2)$ and $(2,3)$ for 
$a=2,3,4$, respectively. For $5\leq a\leq 8$, we shall consider $a$ columns
Padovan tables. Let us begin with $a=5$.%
\begin{equation*}
\left|
\begin{array}{ccccc}
1 & 1 & 2 & 2 & 3 \\ 
4 & 5 & 7 & 9 & 12 \\ 
16 & 21 & 28 & 37 & 49 \\ 
65 & 86 & 114 & 151 & ...%
\end{array}%
\right|.
\end{equation*}%
Then it can be observed that, for instance%
\begin{equation*}
\left\{ 
\begin{array}{c}
P_{16}=65=5\cdot 16-4\cdot 4+1=5P_{11}-4P_{6}+P_{1} \\ 
P_{17}=86=5\cdot 21-4\cdot 5+1=5P_{12}-4P_{7}+P_{2} \\ 
P_{18}=114=5\cdot 28-4\cdot 7+2=5P_{13}-4P_{8}+P_{3}%
\end{array}%
\right. .
\end{equation*}%
Thus, by assuming $P_{t}=5P_{t-5}-4P_{t-10}+P_{t-15}$ for all $t<n$, the
induction hypothesis gives rise to 
\begin{eqnarray*}
P_{n} &=&P_{n-2}+P_{n-3} \\
&=&\left( 5P_{n-7}-4P_{n-12}+P_{n-17}\right) +\left(
5P_{n-8}-4P_{n-13}+P_{n-18}\right) \\
&=&5\left( P_{n-7}+P_{n-8}\right) -4\left( P_{n-12}+P_{n-13}\right) +\left(
P_{n-17}+P_{n-18}\right) \\
&=&5P_{n-5}-4P_{n-10}+P_{n-15},
\end{eqnarray*}%
so $(\rho _{5},\sigma _{5})=(5,-4).$ Moreover from the $6$ columns Padovan
table%
\begin{equation*}
\left|
\begin{array}{cccccc}
1 & 1 & 2 & 2 & 3 & 4 \\ 
5 & 7 & 9 & 12 & 16 & 21 \\ 
28 & 37 & 49 & 65 & 86 & 114 \\ 
151 & 200 & 265 & 351 & 465 & ...%
\end{array}%
\right|
\end{equation*}%
we can observe that, for instance 
\begin{equation*}
\left\{ 
\begin{array}{c}
P_{19}=151=5\cdot 28+2\cdot 5+1=5P_{13}+2P_{7}+P_{1} \\ 
P_{20}=200=5\cdot 37+2\cdot 7+1=5P_{14}+2P_{8}+P_{2} \\ 
P_{21}=265=5\cdot 49+2\cdot 9+2=5P_{15}+2P_{9}+P_{3}%
\end{array}%
\right. .
\end{equation*}%
By assuming $P_{t}=5P_{t-6}+2P_{t-12}+P_{t-18}$ for all $t<n$, we have 
\begin{eqnarray*}
P_{n} &=&P_{n-2}+P_{n-3} \\
&=&\left( 5P_{n-8}+P_{n-14}+P_{n-20}\right) +\left(
5P_{n-9}+2P_{n-15}+P_{n-21}\right) \\
&=&5\left( P_{n-8}+P_{n-9}\right) +2\left( P_{n-14}+P_{n-15}\right) +\left(
P_{n-20}+P_{n-21}\right) \\
&=&5P_{n-6}-4P_{n-12}+P_{n-18},
\end{eqnarray*}%
so $(\rho _{6},\sigma _{6})=(5,2).$ Therefore the observations and
mathematical induction prove that the coefficients $(\rho _{a},\sigma _{a})$
for $a=7,8$ satisfying $$P_{n}=\rho _{a}P_{n-a}+\sigma _{a}P_{n-2a}+P_{n-3a}$$
are equal to $(7,1),$ $(10,-5),$ respectively.
\end{proof}

We note that the subscript $n$ of $P_{n}$ could be negative, for example, in 
$6\ $columns Padovan table, $P_{15}=49=5P_{9}+2P_{3}+P_{-3}$. In general,

\begin{definition}
A sequence $p_{n}$ is called a Padovan type if it satisfies $%
p_{n}=p_{n-2}+p_{n-3}$ with any initials $p_{1},$ $p_{2}$ and $p_{3}$.
\end{definition}

\begin{theorem}
For $0\leq a\leq 8$, let $(\rho _{a},\sigma _{a})$ be the coefficient of the
third order recurrence $P_{n}=\rho _{a}P_{n-a}+\sigma _{a}P_{n-2a}+P_{n-3a}.$
Then,
\begin{enumerate}
\item For $1\leq b\leq 5,$ $\{\rho _{a}\}$ is a Padovan type sequence $\rho
_{b+3}=\rho _{b+1}+\rho _{b}$ with initials $\rho _{1}=0,$ $\rho _{2}=2$ and 
$\rho _{3}=3,$ while $\{\sigma _{a}\}$\ satisfies $\sigma _{b+3}=\sigma
_{b}-\sigma _{b+2}$ with $\sigma _{1}=1,$ $\sigma _{2}=-1$ and $\sigma
_{3}=-2.$
\item Moreover, $\rho _{a}=3P_{a-2}-P_{a-4}$ and $\sigma _{a}=-\rho _{-a}$
for $0\leq a\leq 8$, where $P_{n}$ is the $n$th Padovan number.
\end{enumerate}
\end{theorem}
\begin{proof}
By above theorem, $$\{\rho _{a}\}_{a=1}^{8}=\{0,2,3,2,5,5,7,10\}$$ and $$%
\{\sigma _{a}\}_{a=1}^{8}=\{1,-1,-2,3,-4,2,1,-5\}.$$ So it is easy to see
that $\rho _{b+3}=\rho _{b+1}+\rho _{b}$ and $\sigma _{b+3}=\sigma
_{b}-\sigma _{b+2}$ for $1\leq b\leq 5$. Moreover, by means of Padovan
numbers $P_{a}$, we notice%
\begin{equation*}
\rho _{1}=0=3P_{-1}-P_{-3},\text{ }\rho _{2}=2=3P_{0}-P_{-2},\text{ }\rho
_{3}=3=3P_{1}-P_{-1},
\end{equation*}%
and $\rho _{4}=\rho _{2}+\rho _{1}=3P_{2}-P_{0},$ etc. So $\rho
_{a}=3P_{a-2}-P_{a-4}$ for $1\leq a\leq 8.$ Now, by considering $P_{n}$ with
negative $n$, the Padovan type sequence $\{\rho _{a}\}$ can be extended to
any $a\in \mathbb{Z}$, as follows.
\begin{equation*}
\begin{tabular}{l|llllllllllllll}
\hline
$a$ & $...$ & $-8$ & $-7$ & $-6$ & $-5$ & $-4$ & $-3$ & $-2$ & $-1$ & $0$ & $%
1$ & $2$ & $3$ & $...$ \\ \hline
$\rho _{a}$ & $...$ & $5$ & $-1$ & $-2$ & $4$ & $-3$ & $2$ & $1$ & $-1$ & $3$
& $0$ & $2$ & $3$ & $...$ \\ \hline
\end{tabular}%
\end{equation*}%
Then by comparing $\{\rho _{a}\}_{a=-1}^{-8}=\{-1,1,2,-3,4,-2,-1,5\}$ with $%
\{\sigma _{a}\}_{a=1}^{8}$, we find that $\sigma _{a}=-\rho _{-a}$ for $%
0\leq a\leq 8$.
\end{proof}
\section{The third-order linear recurrence of $P_{n}$}
We shall generalize the findings in above section for $0\leq a\leq 8$ to any
integer $a$.
\begin{theorem}
Let $\rho _{a}=3P_{a-2}-P_{a-4}$ and $\sigma _{a}=-\rho _{-a}$ for any $a\in \mathbb{Z}$. Then, any $n$th Padovan number satisfies $P_{n}=\rho _{a}P_{n-a}+\sigma
_{a}P_{n-2a}+P_{n-3a}$ for every $a<n$.
\end{theorem}
\begin{proof}
It is due to above theorem if $0\leq a\leq 8$. Since $\rho
_{a}=3P_{a-2}-P_{a-4}$ for all $a\in \mathbb{Z}$, $\{\rho _{a}\}$ is a Padovan type sequence because 
\begin{eqnarray*}
\rho _{a+1}+\rho _{a} &=&\left( 3P_{a-1}-P_{a-3}\right) +\left(
3P_{a-2}-P_{a-4}\right) \\
&=&3\left( P_{a-1}+P_{a-2}\right) -\left( P_{a-3}+P_{a-4}\right) \\
&=&3P_{a+1}-P_{a-1}=\rho _{a+3}.
\end{eqnarray*}%
Similarly, since $\sigma _{a}=-\rho _{-a}$ for all $a$, $\{\sigma _{a}\}$
satisfies 
\begin{eqnarray*}
\sigma _{a}-\sigma _{a+2} &=&-\rho _{-a}+\rho _{-(a+2)} \\
&=&-\left( \rho _{-a-2}+\rho _{-a-3}\right) +\rho _{-(a+2)} \\
&=&-\rho _{-(a+3)}=\sigma _{a+3}.
\end{eqnarray*}%
We now suppose that the three order recurrence $P_{n}=\rho
_{t}P_{n-t}+\sigma _{t}P_{n-2t}+P_{n-3t}$ hold for all $t<a$. Since $%
P_{n-(a-2)}=P_{n-a}+P_{n-(a+1)},$ $%
P_{n-2(a-2)}=2P_{n-2(a-1)}-P_{n-2a}+P_{n-2(a+1)}$ and $%
P_{n-3(a-2)}=3P_{n-3(a-1)}-2P_{n-3a}+P_{n-3(a+1)}.$ Then, by lemma, the
mathematical induction with long calculations proves that 
\begin{eqnarray*}
&&\rho _{a+1}P_{n-(a+1)}+\sigma _{a+1}P_{n-2(a+1)}+P_{n-3(a+1)} \\
&=&\left( \rho _{a-1}+\rho _{a-2}\right) P_{n-(a+1)}+\left( \sigma
_{a-2}-\sigma _{a}\right) P_{n-2(a+1)}+P_{n-3(a+1)} \\
&=&\left( \rho _{a-1}+\rho _{a-2}\right) \left( P_{n-(a-2)}-P_{n-a}\right) \\
&&+\left( \sigma _{a-2}-\sigma _{a}\right) \left(
P_{n-2(a-2)}-2P_{n-2(a-1)}+P_{n-2a}\right) \\
&&+\left( P_{n-3(a-2)}-3P_{n-3(a-1)}+2P_{n-3a}\right) \\
&=&P_{n}.
\end{eqnarray*}
\end{proof}

Above theorem provides a good way to find huge Padovan numbers. For
instance, for $40$th Padovan number $P_{40},$ we may choose any $a$, say $%
a=10,$ then $\rho _{10}=3P_{8}-P_{6}=17$ and $\sigma _{10}=-\rho _{-10}=-6$,
thus 
\begin{eqnarray*}
P_{40} &=&\rho _{10}P_{40-10}+\sigma _{10}P_{40-20}+P_{40-30} \\
&=&17\cdot 3329-6\cdot 200+12 \\
&=&55405.
\end{eqnarray*}%
On the other hand, if we take $a=8$ then $\rho _{8}=3P_{6}-P_{4}=10$ and $%
\sigma _{8}=-\rho _{-8}=-5$, so $P_{40}$ can be obtained by $P_{40}=\rho
_{8}P_{32}+\sigma _{10}P_{24}+P_{16}.$

Besides the expression of $\rho _{a}$ by six step apart Padovan numbers,
more identities for $\rho _{a}$ can be developed in terms of three
successive Padovan numbers.

\begin{theorem}
Let $a$ be any integer. Then $\rho
_{a}=P_{a}-P_{a-1}+2P_{a-2}=2P_{a+1}+P_{a}-3P_{a-1}.$ So, $\rho _{a}=\left[ 
\begin{array}{ccc}
2 & 1 & -3%
\end{array}%
\right] \left[ 
\begin{array}{c}
P_{a+1} \\ 
P_{a} \\ 
P_{a-1}%
\end{array}%
\right] .$
\end{theorem}
\begin{proof}
Since 
\begin{eqnarray*}
P_{a-4} &=&P_{a-1}-P_{a-3}=P_{a-1}-\left( P_{a}-P_{a-2}\right) \\
&=&P_{a-1}-P_{a}+P_{a-2},
\end{eqnarray*}%
it follows that 
\begin{eqnarray*}
\rho _{a} &=&3P_{a-2}-P_{a-4} \\
&=&-P_{a-1}+P_{a}+2P_{a-2} \\
&=&-P_{a-1}+P_{a}+2\left( P_{a+1}-P_{a-1}\right) \\
&=&2P_{a+1}+P_{a}-3P_{a-1}.
\end{eqnarray*}%
Hence we have $\rho _{a}=2P_{a+1}+P_{a}-3P_{a-1}=\left[ 
\begin{array}{ccc}
2 & 1 & -3%
\end{array}%
\right] \left[ 
\begin{array}{c}
P_{a+1} \\ 
P_{a} \\ 
P_{a-1}%
\end{array}%
\right] .$

Therefore it follows immediately 
\begin{equation*}
\rho _{-a}=2P_{-a+1}+P_{-a}-3P_{-a-1}=\left[ 
\begin{array}{ccc}
2 & 1 & -3%
\end{array}%
\right] \left[ 
\begin{array}{c}
P_{-a+1} \\ 
P_{-a} \\ 
P_{-a-1}%
\end{array}%
\right] .
\end{equation*}
\end{proof}

For each $n\in 
\mathbb{Z}
$, we define two sequences%
\begin{equation*}
Q_{n}=P_{n}+P_{-n}\text{ and }R_{n}=P_{n}-P_{-n}.
\end{equation*}%
Then it is easy to have the table%
\begin{equation*}
\begin{tabular}{l|llllllllllllll}
\hline
$n$ & $...$ & $1$ & $2$ & $3$ & $4$ & $5$ & $6$ & $7$ & $8$ & $9$ & $10$ & $%
11$ & $12$ & $...$ \\ \hline\hline
$P_{n}$ & $...$ & $1$ & $1$ & $2$ & $2$ & $3$ & $4$ & $\mathbf{5}$ & $7$ & $%
9 $ & $12$ & $16$ & $21$ & $...$ \\ 
$P_{-n}$ & $...$ & $0$ & $1$ & $0$ & $0$ & $1$ & $-1$ & $1$ & $0$ & $-1$ & $%
2 $ & $-2$ & $1$ & $...$ \\ \hline
$Q_{n}$ & $...$ & $1$ & $2$ & $\mathbf{2}$ & $2$ & $\mathbf{4}$ & $\mathbf{3}
$ & $6$ & $7$ & $8$ & $14$ & $14$ & $22$ & $...$ \\ 
$R_{n}$ & $...$ & $1$ & $0$ & $\mathbf{2}$ & $2$ & $\mathbf{2}$ & $\mathbf{5}
$ & $4$ & $7$ & $10$ & $10$ & $18$ & $20$ & $...$ \\ \hline
\end{tabular}%
\end{equation*}

From the table, we notice $P_{7}=5=3+4-2$ or $P_{7}=Q_{6}+Q_{5}-Q_{3}.$
\begin{theorem}
Let $a\in \mathbb{Z}$. Then, the sequences $\{Q_{n}\}$ satisfy the relation $%
P_{n+1}=Q_{n}+Q_{n-1}-Q_{n-3}.$ Furthermore, $P_{n}=\frac{1}{2}\left(
Q_{n}+R_{n}\right) $ and $P_{-n}=\frac{1}{2}\left( Q_{n}-R_{n}\right) .$
\end{theorem}
\begin{proof}
It is easy to see that 
\begin{eqnarray*}
Q_{n} &=&P_{n}+P_{-n} \\
&=&\left( P_{n-2}+P_{n-3}\right) +\left( -P_{-(n-1)}+P_{-(n-3)}\right) \\
&=&Q_{n-3}+P_{n-2}-P_{-(n-1)} \\
&=&Q_{n-3}+\left( P_{n+1}-P_{n-1}\right) -P_{-(n-1)} \\
&=&Q_{n-3}-Q_{n-1}+P_{n+1}.
\end{eqnarray*}%
Hence $P_{n+1}=Q_{n}+Q_{n-1}-Q_{n-3}$.
\end{proof}

\begin{theorem}
Let $n=am+b$ with $1\leq b\leq a<n$. Let $(\rho _{a},\sigma _{a})$ be the
coefficient of the third-order recurrence $P_{n}=\rho _{a}P_{n-a}+\sigma
_{a}P_{n-2a}+P_{n-3a}.$ Then, $P_{n}$ is a linear combination of any three
consecutive entries of $b$th column in the $a$ columns Padovan table.
Furthermore, $P_{n}$ is expressed by the first three terms $P_{2a+b},$ $%
P_{a+b}$ and $P_{b}$ of $b$th column.
\end{theorem}
\begin{proof}
Let $P_{n}=\rho _{a}P_{n-a}+\sigma _{a}P_{n-2a}+P_{n-3a}$ in above theorem.
Then,%
\begin{eqnarray*}
P_{at+b} &=&\rho _{a}P_{a(t-1)+b}+\sigma _{a}P_{a(t-2)+b}+P_{_{a(t-3)+b}} \\
&=&\rho _{a}\left( \rho _{a}P_{a(t-2)+b}+\sigma
_{a}P_{a(t-3)+b}+P_{_{a(t-4)+b}}\right) \\
&&+\sigma
_{a}P_{a(t-2)+b}+P_{_{a(t-3)+b}} \\
&=&\left( \rho _{a}^{2}+\sigma _{a}\right) P_{a(t-2)+b}+\left( \rho
_{a}\sigma _{a}+1\right) P_{a(t-3)+b}+\rho _{a}P_{_{a(t-4)+b}}
\end{eqnarray*}

Continue this process, then it follows that $P_{n}$ is a linear combination
of $P_{2a+b},$ $P_{a+b}$ and $P_{b}$.
\end{proof}

For example, for $P_{38}$ we may take any $a<38$, say $a=7$. Since $(\rho
_{7},\sigma _{7})=(7,1)$, $P_{38}$ can be obtained easily by above theorem that 
\begin{eqnarray*}
P_{38} &=&7P_{31}+P_{24}+P_{17} \\
&=&(7^{2}+1)P_{24}+(7+1)P_{17}+7P_{10} \\
&=&50P_{24}+8P_{17}+7P_{10} \\
&=&50\left( 7P_{17}+P_{10}+P_{3}\right) +8P_{17}+7P_{10} \\
&=&358P_{17}+57P_{10}+50P_{3} \\
&=&358\cdot 86+57\cdot 12+50\cdot 2 \\
&=&\allowbreak 31\,572.
\end{eqnarray*}

However, since $P_{n}$ is composed of $P_{n-a},$ $P_{n-2a}$ and $P_{n-3a}$,
it may be better to choose $a\approx \frac{n}{3}.$ Indeed if we take $\frac{%
56}{3}\approx 18=a$, then $P_{56}=\rho _{18}P_{38}-\rho _{-18}P_{20}+P_{2}$
and the last term $P_{2}$ is known easily.

\begin{remark}
Assume the same context $(\rho _{a},\sigma _{a})$ as before. If $n=3a$, then 
$P_{n}=\rho _{\frac{n}{3}}P_{\frac{2n}{3}}+\sigma _{\frac{n}{3}}P_{\frac{n}{3%
}}+1$ since $P_{0}=1.$ In the other hand, if $n=3a+1$ or $n=3a+2$ and $%
P_{1}=P_{2}=1,$ it follows that 
\begin{equation*}
P_{n}=\rho _{\left\lfloor \frac{n}{3}\right\rfloor }P_{\left\lfloor \frac{2n%
}{3}\right\rfloor +1}+\sigma _{\left\lfloor \frac{n}{3}\right\rfloor
}P_{\left\lfloor \frac{n}{3}+\frac{1}{2}\right\rfloor +1}+1
\end{equation*}
\end{remark}

For example, if $n=26,$ we have $P_{26}=\rho _{8}P_{18}+\sigma
_{8}P_{10}+1=1081.$
\section{Partial Sum of Padovan numbers in a row}
\begin{lemma}
For $m\geq 0$, we have $$\sum_{k=0}^{m}P_{4k}=\frac{1}{5}\left(
P_{4(m+1)}+4P_{4m}+P_{4(m-1)}-1\right).$$
\end{lemma}
\begin{proof}
If $m=0$, we have $P_{0}=1=\frac{1}{5}\left( P_{4}+4P_{0}+P_{-4}-1\right) .$
Now, if $a=4$ and $n=4m$, $P_{4m}=2P_{4(m-1)}+3P_{4(m-2)}+P_{4(m-3)}$. Assume $\sum_{k=0}^{m}P_{4k}=\frac{1}{5}\left(
P_{4(m+1)}+4P_{4m}+P_{4(m-1)}-1\right) $ is true. Then it follows that 
\begin{eqnarray*}
\sum_{k=0}^{m+1}P_{4k} &=&\sum_{k=0}^{m}P_{4k}+P_{4(m+1)} \\
&=&\frac{1}{5}\left( P_{4(m+1)}+4P_{4m}+P_{4(m-1)}-1\right) +P_{4(m+1)} \\
&=&\frac{1}{5}\left( \left( 2P_{4(m+1)}+3P_{4m}+P_{4(m-1)}\right)
+P_{4m}+4P_{4(m+1)}-1\right) \\
&=&\frac{1}{5}\left( P_{4(m+2)}+4P_{4(m+1)}+P_{4m}-1\right) .
\end{eqnarray*}
\end{proof}

\begin{remark}
This theorem is a sum of $4m$ subscripted Padovan numbers. But in our context, it can be explained as a sum of entries of $4$th column in the $4$ columns Padovan table. We now shall study the sum of entries of any $b$th column in the $4$ columns Padovan table.
\end{remark}

Consider $n=am+b$ or $P_{am+b}$ ($m\geq 0$ and $1\leq b\leq a$) as an entry
placed at the $(m+1)$th row and $b$th column in the table, and let%
\begin{equation*}
r_{m}^{(a,b)}=\sum_{k=0}^{m}P_{ak+b}=P_{b}+P_{a+b}+\cdots +P_{am+b}
\end{equation*}%
be the partial sum of $m+1$ entries of $b$th column.

\begin{theorem}
Consider $r_{m}^{(3,b)}$\ with $1\leq b\leq 3$. Then for $m\geq 3$,%
\begin{equation*}
r_{m}^{(3,b)}=3r_{m-1}^{(3,b)}-2r_{m-2}^{(3,b)}+r_{m-3}^{(3,b)}+1.
\end{equation*}
\end{theorem}
\begin{proof}
The $3$ columns Padovan table makes the table of $r_{m}^{(3,b)}$ as follows. 
\begin{equation*}
\left|
\begin{tabular}{lll}
$1$ & $1$ & $2$  \\ 
$2$ & $3$ & $4$  \\ 
$5$ & $7$ & $9$  \\ 
$12$ & $16$ & $21$  \\ 
$28$ & $37$ & $\cdots $%
\end{tabular}%
\right| \text{ and 
\begin{tabular}{l|lll}
\hline
$m$ & $r_{m}^{(3,1)}$ & $r_{m}^{(3,2)}$ & $r_{m}^{(3,3)}$ 
\\ \hline
$0$ & $1$ & $1$ & $2$  \\ 
$1$ & $3$ & $4$ & $6$  \\ 
$2$ & $8$ & $11$ & $15$  \\ 
$3$ & $20$ & $27$ & $36$  \\ 
$4$ & $48$ & $64$ & $85$  \\ \hline
\end{tabular}%
.}
\end{equation*}%
When $m=3$, we notice $20=\left( 3\cdot 8-2\cdot 3+1\right) +1$, and it
can be written as $r_{3}^{(3,1)}=3r_{2}^{(3,1)}-2r_{1}^{(3,1)}+r_{0}^{(3,1)}+1$. Similarly for $b=2,3$.

Assume that $%
r_{m}^{(3,1)}=3r_{m-1}^{(3,1)}-2r_{m-2}^{(3,1)}+r_{m-3}^{(3,1)}+1.$ Then by
induction hypothesis yields %
\begin{eqnarray*}
r_{m+1}^{(3,1)}&=&\left(3r_{m-1}^{(3,1)}-2r_{m-2}^{(3,1)}+r_{m-3}^{(3,1)}+1\right)+P_{3(m+1)+1} \\
&=&3r_{m-1}^{(3,1)}-2r_{m-2}^{(3,1)}+r_{m-3}^{(3,1)}+1+
3P_{3m+1}-2P_{3(m-1)+1}+P_{3(m-2)+1} \\
&=&3r_{m}^{(3,1)}-2r_{m-1}^{(3,1)}+r_{m-2}^{(3,1)}+1,
\end{eqnarray*}%
this proves the equation if $b=1$. The other cases are similar.
\end{proof}

\begin{theorem}
Consider $r_{m}^{(4,b)}$\ with $1\leq b\leq 4$. Then for $m\geq 3$,%
\begin{equation*}
r_{m}^{(4,b)}=\left\{ 
\begin{array}{ccc}
2r_{m-1}^{(4,b)}+3r_{m-2}^{(4,b)}+r_{m-3}^{(4,b)}+2 & if & b=1, \\ 
2r_{m-1}^{(4,b)}+3r_{m-2}^{(4,b)}+r_{m-3}^{(4,b)}+4 & if & b=2, \\ 
2r_{m-1}^{(4,b)}+3r_{m-2}^{(4,b)}+r_{m-3}^{(4,b)}+3 & if & b=3, \\ 
2r_{m-1}^{(4,b)}+3r_{m-2}^{(4,b)}+r_{m-3}^{(4,b)}+6 & if & b=4.%
\end{array}%
\right.
\end{equation*}
\end{theorem}
\begin{proof}
The $4$ columns Padovan table makes the table of $r_{m}^{(4,b)}$ as follows. 
\begin{equation*}
\left|
\begin{tabular}{llll}
$1$ & $1$ & $2$ & $2$ \\ 
$3$ & $4$ & $5$ & $7$ \\ 
$9$ & $12$ & $16$ & $21$ \\ 
$28$ & $37$ & $49$ & $65$ \\ 
$86$ & $114$ & $151$ & $\cdots $%
\end{tabular}%
\right| \text{ and 
\begin{tabular}{l|llll}
\hline
$m$ & $r_{m}^{(4,1)}$ & $r_{m}^{(4,2)}$ & $r_{m}^{(4,3)}$ & $r_{m}^{(4,4)}$
\\ \hline
$0$ & $1$ & $1$ & $2$ & $2$ \\ 
$1$ & $4$ & $5$ & $7$ & $9$ \\ 
$2$ & $13$ & $17$ & $23$ & $30$ \\ 
$3$ & $41$ & $54$ & $72$ & $95$ \\ 
$4$ & $127$ & $168$ & $223$ & $295$ \\ \hline
\end{tabular}%
.}
\end{equation*}%
When $m=4$, we notice $127=\left( 2\cdot 41+3\cdot 13+4\right) +2$, and it
can be written as 
\begin{equation*}
r_{4}^{(4,1)}=2r_{3}^{(4,1)}+3r_{2}^{(4,1)}+r_{1}^{(4,1)}+2.
\end{equation*}%
Similar to this, we observe that 
\begin{equation*}
\left\{ 
\begin{array}{c}
2r_{3}^{(4,2)}+3r_{2}^{(4,2)}+r_{1}^{(4,2)}+4 \\ 
2r_{3}^{(4,3)}+3r_{2}^{(4,3)}+r_{1}^{(4,3)}+3 \\ 
2r_{3}^{(4,4)}+3r_{2}^{(4,4)}+r_{1}^{(4,4)}+6%
\end{array}%
\right. .
\end{equation*}%
Assume that $%
r_{m}^{(4,1)}=2r_{m-1}^{(4,1)}+3r_{m-2}^{(4,1)}+r_{m-3}^{(4,1)}+2.$ Then by
induction hypothesis yields %
\begin{eqnarray*}
r_{m+1}^{(4,1)}&=&\sum_{k=0}^{m}P_{4k+1}+P_{4(m+1)+1} \\
&=&2r_{m-1}^{(4,1)}+3r_{m-2}^{(4,1)}+r_{m-3}^{(4,1)}+2+P_{4(m+1)+1} \\
&=&2r_{m-1}^{(4,1)}+3r_{m-2}^{(4,1)}+r_{m-3}^{(4,1)}+2+
2P_{4m+1}+3P_{4(m-1)+1}+P_{4(m-2)+1} \\
&=&2r_{m}^{(4,1)}+3r_{m-1}^{(4,1)}+r_{m-2}^{(4,1)}+2,
\end{eqnarray*}%
this proves the first equation. Similarly, other relationships are followed.
\end{proof}

\end{document}